\documentclass{cmslatex}
\usepackage[paperwidth=7in, paperheight=10in, margin=.875in]{geometry}
 \usepackage[backref,colorlinks,linkcolor=red,anchorcolor=green,citecolor=blue]{hyperref}
\usepackage{amsfonts,amssymb}
\usepackage{amsmath}
\usepackage{graphicx}
\usepackage{cite}
\usepackage{enumerate}
\usepackage{bbm}
\usepackage{graphicx}

\newcommand{\xdownarrow}[1]{%
	{\left\downarrow\vbox to #1{}\right.\kern-\nulldelimiterspace}
}
\renewcommand{\theequation}{\arabic{section}.\arabic{equation}}

   \allowdisplaybreaks
\begin{document}
 \title{Well-posedness of a coupled system of  Skorohod-like stochastic differential equations\thanks{Received date, and accepted date (The correct dates will be entered by the editor).}}


          \author{Thi Kim Thoa Thieu\thanks{Department of Mathematics and Computer Science, Karlstad University, Karlstad, Sweden, (ttkthoa93@gmail.com).}
          \and Adrian Muntean\thanks{Department of Mathematics and Computer Science  \& Center for Societal Risk Research (CSR),  Karlstad University, Karlstad, Sweden, (adrian.muntean@kau.se).}}

         \pagestyle{myheadings} \markboth{Well-posedness of a coupled system of Skorohod-like stochastic differential equations}{Thoa Thieu and Adrian Muntean} \maketitle

          \begin{abstract}
               We study the well-posedness of a coupled system of Skorohod-like stochastic differential equations with reflecting boundary condition. The setting describes the  evacuation dynamics of a mixed crowd composed of both active and passive pedestrians moving through a domain with obstacles, fire and smoke. As main working techniques, we use compactness methods  and the Skorohod's representation of solutions to SDEs posed in bounded domains. This functional setting is a new point of view in the field of modeling and simulation pedestrian dynamics. The main challenge is to handle the coupling in the model equations together with the multiple-connectedness of the domain and the pedestrian-obstacle interaction. 
          \end{abstract}

          \section{Introduction}\label{intro}

In this paper, we study the well-posedness of a coupled system of  Skorohod-like stochastic differential equations modeling the dynamics of pedestrians  through a heterogenous domain in the presence of fire. From the modeling perspective, our approach is novel, opening many routes for investigation especially what concerns the computability of solutions  and identification of model parameters.  The standing assumption is that the crowd of pedestrians is composed of two distinct populations: an {\em active population}  -- these pedestrians are aware of the details of the environment and move towards the exit door, and a {\em passive population} -- these pedestrians are not aware of the details of the geometry and move randomly to explore the environment and eventually to find the exit. All pedestrians are seen as moving point particles driven by a suitable over-damped Langevin model, which will be described in Section \ref{Model}. Our model belongs to the class of social-velocity models for crowd dynamics.  It is posed in a two dimensional multiple connected region $D$, containing obstacles with a fixed location. Furthermore,  a stationary fire, which produces smoke, is placed within the geometry forcing the pedestrians to choose a proper own velocity such that they evacuate. The fire is seen, as a first attempt, as a stationary obstacle. 

To keep a realistic picture, the overall dynamics is restricted to a bounded "perforated" domain, i.e. all obstacles are seen as impenetrable regions. The geometry is described in Subsection \ref{geometry};  see Figure \ref{fig:fig1} to fix ideas.
In this framework, we consider  reflecting boundary conditions and plan, as further research, to treat the case of  mixed reflection--flux boundary conditions so that the exits can allow for outflux.  In this framework, we focus on the interior obstacles. To achieve a correct dynamics of dynamics of the pedestrians close to the boundary of the interior obstacles, we choose to work with the classical Skorohod's formulation of SDEs; we refer the reader to the textbook \cite{Pilipenko2014} for more details on this subject. Note that this approach is needed especially because of the chosen dynamics for the passive and active pedestrians who are able to avoid collisions with the obstacles by using a motion planning  map ({\em a priori} given paths -- solution to a suitable Eikonal-like equation; cf. Appendix \ref{Eikonal}).

\section{Related contributions. Main questions of this research}

A number of relevant results are available on the dynamics of mixed active-passive pedestrian populations. As far as we are aware, the first questions in this context were posed in the modeling and simulation study \cite{Richardson2019}  while considering the evacuation dynamics of a mixed active--passive pedestrian populations in a complex geometry in the presence of a fire as well as of a slowly spreading smoke curtain. From a stochastic processes perspective, various lattice gas models for active-passive pedestrian dynamics have been recently explored in \cite{Cirillo2019} and \cite{Colangeli2019}. See also \cite{Thieu2019} for a result on  the weak solvability of a deterministic system of parabolic partial differential equations describing the interplay of a mixture of fluids for active--passive populations of pedestrians. A mean-field approach to the pedestrian dynamics has been reported in \cite{Aurell2020}, where the authors provided a system of SDEs of mean-field type modeling pedestrian motion. In particular, their system of SDEs describes a scenario where pedestrians spend time at and move along walls by means of sticky boundaries and boundary diffusion. 

The discussion of the active-passive pedestrian dynamics at the level of SDEs is new and brings in at least a twofold challenge: (i) the evolution system is coupled and (ii) pedestrians have to cross a domain with forbidden regions (the obstacles). Various solution strategies have been already identified for deterministic crowd  evolution equations. We mention here the two ideas which stand out: a granular media approach, where collisions with obstacles are tackled with techniques of non-smooth analysis cf. e.g. \cite{Maury2015}, and a reflection-of-velocities approach as it is done e.g. in \cite{Kimura2019}.  If some level of noise affects the dynamics, then both these approaches fail to be aplicable. On the other hand, there are several results for stochastic differential equations with reflecting boundary conditions, one of them being the seminal  contribution of Skorohod in \cite{Skorohod1961}, where the author provided the existence and uniqueness to one dimensional stochastic equations for diffusion processes in a bounded region. A direct approach to the solution of the reflecting boundary conditions and reductions to the case including nonsmooth ones are reported in \cite{Lions1984}. Extending results by Tanaka, the author of \cite{Saisho1987}  proves the existence and uniqueness of solutions to the Skorohod equation posed in a bounded domain in $\mathbb{R}^d$ where a reflecting boundary condition is applied. In \cite{Dupuis1993}, the authors studied a strong existence and uniqueness to the stochastic differential equations with reflecting boundary conditions for domains that might have conners. In addition, the existence, uniqueness and stability of solutions of multidimensional SDE's with reflecting boundary conditions has been provided in \cite{Slominski1993}, where the author obtained results on the existence and uniqueness of strong and weak solutions to the SDE for any driving semimartingale and in a more general domain. 

The main question we ask in this paper is: Can we frame our crowd dynamics model as a well-posed system of stochastic evolution equations of Skorohod type? Provided suitable restrictions on the geometry of the domain,  on the structure of nonlinearites as well as data and parameters, we provide in Section \ref{main} a positive answer to this question. This study opens the possibility of exploring further our system from the numerical analysis perspective so that suitable  algorithms can be designed to produce simulations forecasting the evacuation time based on our model. A couple of follow-up open questions are given in the conclusion; see Section \ref{conclusion}.

\section{Setting of the model equations}\label{Model}
\subsection{Geometry}\label{geometry}

\begin{figure}[h!]
	\centering
	\includegraphics[width=1.0\textwidth]{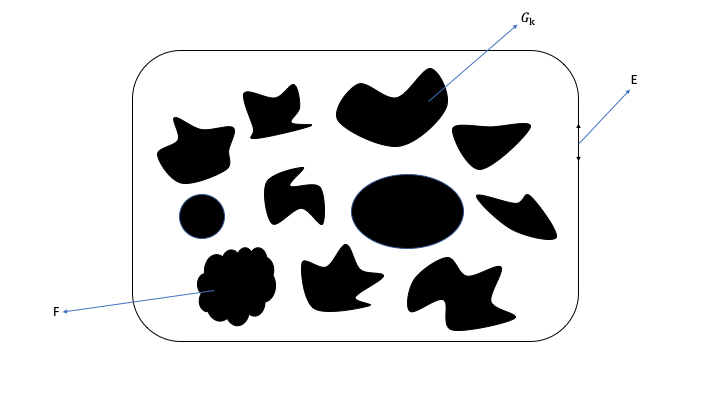}
	\caption{Basic geometry for our active-passive pedestrian model. Initially, pedestrians occupy some random position within a geometry with obstacles $G_k$. Because of the presence of the fire  $F$, and presumably also of smoke, they wish  to evacuate via the exit door $E$ while avoiding the obstacles $G_k$ and the fire $F$.}
	\label{fig:fig1}
\end{figure}
We consider a two dimensional domain, which we refer to as $\Lambda$.  As a building geometry, parts of the domain are filled with obstacles. Their collection is denoted by $G = \bigcup_{k=1}^{N_{\text{obs}}} G_k$, for all $k \in \{1, \ldots, N_{\text{obs}} \in \mathbb{N}\}$. A fire $F$ is introduced somewhere in this domain and is treated in this context as an obstacle for the motion of the crowd. Moreover, the domain has the exit denoted by $E$. Our domain represents the environment where the crowd of pedestrians is located. The crowd tries to find the fastest way to the exit, avoiding the obstacles and the fire.
Let $D:=\Lambda\backslash(G \cup E \cup F) \subset \mathbb{R}^2$ with the boundary $\partial D$ such that $\partial \Lambda \cap \partial G_k = \emptyset$, $\partial \Lambda \cap \partial G_k = \emptyset$ and $F \cap G_k=\emptyset$, see in Section \ref{assumption} for our assumption of the regularity of the involve set. We also denote $S=(0,T)$ for some $T \in \mathbb{R}_{+}$. We refer to $\bar{D}\times S$ as $D_T$, note that $\bar{D}$ denotes the closure of $D$. Furthermore, $N_A$ is the total number of active agents, $N_P$ is the total number of passive particles with $N:=N_A+N_P$ and $N_A, N_P, N \in \mathbb{N}$. 

\subsection{Active population}\label{ap}

For $i \in \{1,\ldots, N_A\}$ and $t \in S$, let $\textbf{x}_{a_i}$ denote the position of the pedestrian $i$ belonging to the active population at time $t$. We assume that the dynamics of active pedestrians is governed by  
\begin{align}\label{active}
\begin{cases}
\frac{d \textbf{x}_{a_i}(t)}{dt} = -\Upsilon(s(\textbf{x}_{a_i}(t))\frac{\nabla \phi(\textbf{x}_{a_i}(t))}{|\nabla \phi(\textbf{x}_{a_i}(t))|}\left(p_{\text{max}} - p(\textbf{x}_{a_i}(t), t)\right),\\
\textbf{x}_{a_{i}}(0) = \textbf{x}_{a_{i0}},
\end{cases}
\end{align}
where $\textbf{x}_{a_{i0}}$ represents the initial configuration of active pedestrians inside $\bar{D}$. In \eqref{active}, $\nabla \phi$ is the minimal motion path of the distance between particle positions $\textbf{x}_{a_i}$ and the exit location $E$ (it solves the Eikonal-like equation). The function $\phi(\cdot)$ encodes the familiarity with the geometry; see also \cite{Yang2016} for a related setting. We refer to as the motion planning map.  In this context,  $p(\textbf{x},t)$ is the local discomfort (a realization of the social pressure) so that  
\begin{align}\label{pressure}
p(\textbf{x}(t),t)&=\mu(t)\int_{D \cap B(\textbf{x}(t),\tilde{\delta})} \sum_{j=1}^{N}\delta(y-\textbf{x}_{c_j}(t))dy,
\end{align}
for $\{\textbf{x}_{c_j}\}:= \{\textbf{x}_{a_i}\} \cup \{\textbf{x}_{b_k}\}$ for $i \in \{1, \dots, N_A\}, k \in \{1, \dots, N_P\}, j \in \{1, \ldots, N_A + N_P\}$ and $\textbf{x}_{b_k}$ represents the position of the pedestrian $k$ belonging at time $t$ to the passive population.
In \eqref{pressure}, $\delta$ is the Dirac (point) measure, $\mu(t)$ is a finite measure and $B(\textbf{x},\tilde{\delta})$ is a ball center $\textbf{x}$ with small enough radius $\tilde{\delta}$ such that $\tilde{\delta} >0$. Hence, the discomfort $p(\textbf{x},t)$ represents a finite measure on the set $D \cap B(\textbf{x},\tilde{\delta})$. In addition, we assume the following structural relation between the smoke extinction and the walking speed (see in \cite{Jin1997} and \cite{Ronchi2019}) as a function 
$\Upsilon: \mathbb{R}_{+}^2 \longrightarrow \mathbb{R}^2$ such that
\begin{align}
\Upsilon(\textbf{x}) &= -\zeta \textbf{x} + \eta,
\end{align}
where $\zeta, \eta$ are given real positive numbers.
The dependence of the model coefficients on the local smoke density is marked via a smooth relationship with respect to an  a priori given function $s(\textbf{x}(t))$ describing the distribution of smoke inside the geometry at position $x$ and time $t$.

\subsection{Passive population}

For $k \in \{1,\ldots, N_P\}$ and $t \in S$, let $\textbf{x}_{p_k}$ denote the position of the pedestrian $k$ belonging at time $t$ to the passive population. The dynamics of the passive pedestrians is described  here as a system of stochastic differential equations as follows:
\begin{align}\label{passive}
\begin{cases}
d \textbf{x}_{p_k}(t) = \sum_{j=1}^N \frac{\textbf{x}_{c_j} - \textbf{x}_{p_k}}{\epsilon + |\textbf{x}_{c_j} - \textbf{x}_{p_k}|}\omega(|\textbf{x}_{c_j} - \textbf{x}_{p_k}|, s(\textbf{x}_{p_k}(t)))dt + \beta(s(\textbf{x}_{p_k}(t)))dB(t),\\
\textbf{x}_{p_{k}}(0) = \textbf{x}_{p_{k0}},
\end{cases}
\end{align}
where $\textbf{x}_{p_{k0}}$ represents the initial configuration of passive pedestrians inside $\bar{D}$ and $\epsilon>0$. In \eqref{passive}, $\omega$ is a Morse-like potential function \footnote{In general, the interactions between particles are defined in the senses of a pairwise potential, i.e. repulsive at short ranges and attractive at longer ranges. Therefore, one of the typical choices is the exponentially decaying Morse potential, which is known to reproduce certain types of collective motion observed in nature, particularly aligned flocks and rotating mills (see e.g. in \cite{Bernoff2011}  and \cite{Carrillo2013}).} (see e.g. Ref. \cite{Carrillo2013} for a setting where a similar potential has been used). We take $\omega: \mathbb{R}\times \mathbb{R}^2 \longrightarrow \mathbb{R}^2$ to be
\begin{align}
\omega(x,y) = -\beta(y) \left(-C_A e^{-\frac{x}{\ell_A} }+ C_R e^{-\frac{x}{\ell_R}}\right) \text{ for } x, y \in \mathbb{R}\times \mathbb{R}^2,
\end{align} 
while $C_A >0, C_R >0$ are the attractive and repulsive strengths and $\ell_A >0, \ell_R>0$ are the respective length scales for attraction and repulsion. Moreover, the coefficient $\beta$ is a regularized version of the Heaviside step function. As in Subsection \ref{ap}, the dependence of the model coefficients on the smoke is marked via a smooth relationship with respect to an  a priori given function $s(x,t)$ describing the distribution of smoke inside the geometry at position $x$ and time $t$.
Note that the passive  pedestrians do not posses any knowledge on the geometry of the walking space. In particular, the location of the exit is unknown; see \cite{CristianiD2019} for a somewhat related context.

\section{Technical preliminaries and assumptions}\label{pre-assumptions}
\subsection{Technical preliminaries}

We recall the classical Ascoli-Arzel\`{a} Theorem:\\
A family of functions $U \subset C(\bar{S};\mathbb{R}^d)$ is relatively compact (with respect to the uniform topology) if
\begin{enumerate}
	\item[i.] for every $t \in \bar{S}$, the set $\{f(t);f \in U\}$ is bounded.
	\item[ii.] for every $\varepsilon>0$ and $t,s \in \bar{S}$, there is $\bar{\delta} > 0$ such that
	\begin{eqnarray}
	|f(t)-f(s)| \leq \varepsilon,
	\end{eqnarray}
	whenever $|t - s| \leq \bar{\delta}$ for all $f \subset U$.
\end{enumerate}
For a function $f: \bar{S} \to \mathbb{R}^d$,  we introduce the definition of H\"{o}lder seminorms as
\begin{eqnarray}\label{holder_norm}
[f]_{C^\alpha(S; \mathbb{R}^d)} = \sup_{t\neq s; t,s \in \bar{S}}\frac{|f(t)-f(s)|}{|t-s|^{\alpha}}, 
\end{eqnarray}
for $\alpha \in (0,1)$ and the supremum norm as
\begin{eqnarray}\label{infty_norm}
\|f\|_{L^\infty(S;\mathbb{R}^d)} = \mathrm{ess}\sup_{t \in \bar{S}}|f(t)|.
\end{eqnarray}
We refer to \cite{Adams03} and \cite{Gilbarg1977} for more details on the used function spaces.

Using  Arzel\`{a}-Ascoli Theorem based on the facts:
\begin{enumerate}
	\item[i'.] there is $M_1>0$ such that $\|f\|_{L^\infty(S;\mathbb{R}^d)}\leq M_1$ for all $f \in U$,
	\item[ii'.] for some $\alpha \in (0,1)$, there is an $M_2>0$ such that $[f]_{C^\alpha(\bar{S}; \mathbb{R}^d)} \leq M_2$ for all $f \in U$,
\end{enumerate}
we infer that the set
\begin{eqnarray}\label{relativelycompactKMP}
K_{M_1M_2} = \left\{f\in C(\bar{S};\mathbb{R}^d); \|f\|_{L^\infty(S;\mathbb{R}^d)}\leq M_1, [f]_{C^\alpha(\bar{S}; \mathbb{R}^d)} \leq M_2\right\}
\end{eqnarray}
is relatively compact in $C(\bar{S}; \mathbb{R}^d)$.

For $\alpha \in (0,1)$, $T>0$ and $p>1$, the space 
$W^{\alpha,p}(S;\mathbb{R}^d)$ is defined as the set of all $f \in L^p(S;\mathbb{R}^d)$ such that
\begin{eqnarray}
[f]_{W^{\alpha,p}(S;\mathbb{R}^d)}:= \int_{0}^{T}\int_{0}^{T}\frac{|f(t) - f(s)|^p}{|t-s|^{1 + \alpha p}}dtds < \infty. \nonumber
\end{eqnarray}
This space is endowed with the norm
\begin{eqnarray}
\|f\|_{W^{\alpha,p}(S;\mathbb{R}^d)} = \|f\|_{L^p(S;\mathbb{R}^d)} + [f]_{W^{\alpha,p}(S;\mathbb{R}^d)}.\nonumber
\end{eqnarray}

Moreover, we have the following embedding
\begin{eqnarray}
W^{\alpha,p}(S; \mathbb{R}^d) \subset C^\gamma(\bar{S};\mathbb{R}^d) \quad \textrm{ for } \alpha p - \gamma >1\nonumber
\end{eqnarray}
and $[f]_{C^\gamma(S; \mathbb{R}^d)} \leq C_{\gamma,\alpha,p}\|f\|_{W^{\alpha,p}(S; \mathbb{R}^d)}$. Relying on the Ascoli-Arzel\`{a} Theorem, we have the following situation:
\begin{enumerate}
	\item[ii''.] for some $\alpha \in (0,1)$ and $p>1$ with $\alpha p>1$, there is $M_2>0$ such that $[f]_{W^{\alpha,p}(S;\mathbb{R}^d)} \leq M_2$ for all $f \in U$. 
\end{enumerate}
If $\textrm{i'}$ and $\textrm{ii''}$ hold, then the set
\begin{eqnarray}\label{relativelycompactKMP_2}
K'_{M_1M_2} = \left\{f \in C(\bar{S}; \mathbb{R}^d); \|f\|_{L^\infty(S;\mathbb{R}^d)} \leq M_1, [f]_{W^{\alpha,p}(S;\mathbb{R}^d)} \leq M_2\right\}
\end{eqnarray}
is relatively compact in $C(\bar{S}; \mathbb{R}^d)$, if $\alpha p>1$ (see e.g. \cite{Flandoli95}, \cite{Colangeli2019}).

\subsection{Assumptions}\label{assumption}

To be successful with our mathematical analysis, we rely on the following assumptions: 

\begin{itemize}
	\item[($\text{A}_1$)] The functions $b: D_T\times D_T \longrightarrow \mathbb{R}^2\times \mathbb{R}^2,$ and $\sigma: D_T \times D_T \longrightarrow \mathbb{R}^{2\times 2} \times \mathbb{R}^{2\times 2}$ satisfy 
	$|\sigma(x, t)| \leq L$, $|b(x, t)| \leq L$ for all $x \in \bar{D} \times \bar{D}$ and $t \in S$. Here $\sigma$ and $b$ incorporate the right-hand sides of the SDEs (\ref{active}) and (\ref{passive}) in their respective dimensionless form indicated in Appendix   \ref{nondim}.
	\item[($\text{A}_2$)] $p_{\text{max}} = N|\bar{D}|$, where $|\bar{D}|$ denotes the area of $\bar{D}$.
	\item[($\text{A}_3$)] $\Upsilon, \omega, \beta\in C^1( \mathbb{R}^2)$.
	\item[($\text{A}_4$)] $s \in  C^1(\bar{S}; \mathbb{R}^2)$.
	\item[($\text{A}_5$)] $\partial D$ is $C^{2,\alpha}$ with $\alpha \in (0,1)$. 
\end{itemize}

It is worth mentioning that assumptions ($\text{A}_1$) and ($\text{A}_2$) correspond to the modeling of the situation, while ($\text{A}_3$)-($\text{A}_5$) are of technical nature, corresponding to the type of solution we are searching for; clarifications in this direction are given in the next Section.
\section{The Skorohod equation}\label{Skorohod}
\subsection{Concept of solution}

Take $x \in \partial D$ arbitrarily fixed. We define the set $\mathcal{N}_x$ of inward normal unit vectors at $x \in \partial D$ by
\begin{align}
\mathcal{N}_{x}&= \cup_{r > 0}\mathcal{N}_{x,r}, \nonumber\\
\mathcal{N}_{x,r}&=\left\{\textbf{n} \in \mathbb{R}^2: |\textbf{n}|=1, B(x-r\textbf{n},r) \cap D = \emptyset \right\},
\end{align}
where $B(z,r) = \{y \in \mathbb{R}^2: |y-z| < r\}, z \in \mathbb{R}^2, r>0$. Mind that, in general, it can happen that $\mathcal{N}_x = \emptyset$. In this case, the uniform exterior sphere condition is not satisfied (see, for instance,  the examples in \cite{Cholaquidis2016}, Fig. 5 and in \cite{Choulli2016}, page $4$). 
\newline
We complement our list of  assumptions  ($\text{A}_1$)--($\text{A}_5$) with three specific conditions on the geometry of the domain $D$:

\begin{itemize}
	\item[($\text{A}_6$)] (Uniform exterior sphere condition). There exists a constant $r_0 > 0$ such that
	$$\mathcal{N}_x = \mathcal{N}_{z,r_0} \neq \emptyset \text{ for any } z \in \partial D.$$
	\item[($\text{A}_7$)] There exits constants $\delta > 0$ and $\delta' \in [1, \infty)$ with the following property: for any $x \in \partial D$ there exists a unit vector $\textbf{l}_x$ such that 
	$$\langle\textbf{l}_x,\textbf{n}\rangle \geq 1/\delta' \text{ for any } \textbf{n} \in \bigcup_{y \in B(x, \delta) \cap \partial D} \mathcal{N}_y,$$ 
	where $\langle\cdot, \cdot\rangle$ denotes the usual inner product in $\mathbb{R}^2$.
	\item[($\text{A}_8$)]  There exist $\delta''>0$ and $\nu>0$ such that for each $x_0 \in \partial D$ we can find a function $f \in C^2(\mathbb{R}^2)$ satisfying 
	\begin{align}
	\langle y-x, \textbf{n}\rangle + \frac{1}{\nu}\langle \nabla f(x), \textbf{n}\rangle|y-x|^2 \geq 0,
	\end{align}
	for any $x \in B(x_0, \delta'') \cap \partial D, y \in B(x_0, \delta'') \cap \partial \bar{D}$ and $\textbf{n} \in \mathcal{N}_x$.
\end{itemize}

The following relation is called {\em the Skorohod equation}: Find $(\xi, \phi) \in C(\bar{S}, \mathbb{R}^2)$ such that
\begin{align}\label{deterministicS}
\xi(t)= w(t) + \phi(t),
\end{align}
where $w \in C(\bar{S}, \mathbb{R}^2)$ is given so that $w(0) \in \bar{D}$. 
The solution of \eqref{deterministicS} is a pair $(\xi,\phi)$, which satisfies the following two conditions:
\begin{itemize}
	\item[(a)] $\xi \in C(\bar{S}, \bar{D})$; 
	\item[(b)] $\phi \in C(\bar{S})$ with bounded variation on each finite time interval satisfying $\phi(0)=0$ and 
	\begin{align}
	\phi(t) &= \int_0^t\textbf{n}(y)d|\phi|_y,\nonumber\\
	|\phi|_{t} &= \int_0^t\mathbbm{1}_{\partial D}(\xi(y)) d|\phi|_y,
	\end{align}
	where 
	\begin{align}\label{BV-func}
	\textbf{n}(y) &\in \mathcal{N}_{\xi(y)} \text{ if } \xi(y) \in \partial D,\\
	|\phi|_t &= \text{ total variation of } \phi \text{ on } [0,t] \nonumber\\
	&=\sup_{\mathcal{T} \in \mathcal{G}([0,t])} \sum_{k=1}^{n_{\mathcal{T}}} |\phi(t_k) - \phi(t_k-1)|.
	\end{align}
	In \eqref{BV-func}, we denote by $\mathcal{G}([0,t])$ the family of all partitions of the interval $[0,t]$ and take a partition $\mathcal{T} = \{0=t_0<t_1< \ldots <t_n=t\} \in \mathcal{G}([0,t])$. The supremum in \eqref{BV-func} is taken over all partitions of type $0=t_0<t_1< \ldots <t_n=t$. 
\end{itemize}
Conditions (a) and (b) guarantee that $\xi$ is a \textit{reflecting process} on $\bar{D}$.

It is easily seen from the definition that
$$\xi_1(t) = w_1(t), \ldots, \xi_{d-1}(t) = \xi_{d-1},$$
and 
$$\xi_d(t)) = w_d(t)+\phi(t), \int_0^t\mathbbm{1}_{\xi_d(y) \notin \partial D}d|\phi|_y.$$
We define a multidimensional Skorohod's map $\Gamma: C(\bar{S}) \longrightarrow C(\bar{S})$ such that
\begin{align}
\Gamma w(t) = \Gamma(w_1, \ldots, w_d)(t) = (w_1(t), \ldots, w_{d-1}, \Gamma w_d(t)).
\end{align}
Hence, the pair $(\xi_d, \phi)$ is the exact solution of the one-dimensional Skorohod problem $\xi_d$. Therefore, it holds
\begin{align}\label{reflection-func}
\phi(t) = -\min_{y \in [0,t]}\{w_d(y),0\}, \quad \ \xi_d(t) =  w_d(t) - \min_{y \in [0,t]}\{w_d(y),0\} = \Gamma w_d(t).
\end{align}
The multidimensional Skorohod's map $\Gamma$ satisfies the Lipschitz condition in a space of continuous functions.
\begin{theorem}\label{deterministic-theo}
	Assume conditions ($\text{A}_6$) and ($\text{A}_7$). Then for any $w \in C(\bar{S}, \mathbb{R}^2)$ with $w(0) \in \bar{D}$, there exists a unique solution $\xi(t,w)$ of the equation \eqref{deterministicS} such that $\xi(t,w)$ is continuous in $(t,w)$.
\end{theorem}
For the proof of this Theorem, we refer the reader to Theorem $4.1$ in \cite{Saisho1987}.

To come closer to the model equations for active-passive pedestrian dynamics described in Section \ref{Model},  we introduce the mappings
$$b: D_T \times D_T \longrightarrow \mathbb{R}^2 \times \mathbb{R}^2 , \quad \sigma: D_T \times D_T \longrightarrow \mathbb{R}^{2\times 2} \times \mathbb{R}^{2\times 2}$$
and consider the Skorohod-like system on the probability space $(\Omega, \mathcal{F}, P)$
\begin{align}\label{main-skorohod}
dX_t = b(X_t(t))dt + \sigma(X_t(t))dB(t) + d\Phi_t.
\end{align}
Note that \eqref{main-skorohod} can be written component-wise as
\begin{align*}
dX_t^{(I)} = b(X_t(t))_{I}dt + \sum_{J=1}^2\sigma_{IJ}(X_t(t))dB^{(J)}(t) + d\Phi_t^{(I)} \text{ for } 1 \leq I \leq 4.
\end{align*}
with
\begin{align}\label{initial-cond}
X(0)=X_0 \in \bar{D},
\end{align}
where the inital value $X_0$ is assumed to be an $\mathcal{F}_0-$measurable random variable and $B(t)$ is a $2-$dimensional $\mathcal{F}_t-$Brownian motion with $B(0)=0$. Here, $\{\mathcal{F}_t\}$ is a filtration such that $\mathcal{F}_0$ contains all $P-$negligible sets and $\mathcal{F}_t= \cap_{\varepsilon>0}\mathcal{F}_{t+\varepsilon}$. The structure of \eqref{main-skorohod} is provided in Section \ref{proof-theo1}. Similarly to the deterministic case, we can now define the following concept of solutions to \eqref{main-skorohod}. More details of the structure of \eqref{main-skorohod}-\eqref{initial-cond}  are listed in Section \ref{proof-theo1}.
\begin{defn}\label{def}
	A pair $(X_t,\Phi_t)$ is called solution to \eqref{main-skorohod}--\eqref{initial-cond} if the following conditions hold:
	
	\begin{itemize}
		\item[(i)] $X_t$ is a $\bar{D}-$valued $\mathcal{F}_t-$adapted continuous process;
		\item[(ii)] $\Phi(t)$ is an $\mathbb{R}^2-$valued $\mathcal{F}_t-$adapted continuous process with bounded variation on each finite time interval such that $\Phi(0)=0$ with
		\begin{align}
		\Phi(t) &= \int_0^t\textbf{n}(y)d|\Phi|_y,\nonumber\\
		|\Phi|_{t} &= \int_0^t\mathbbm{1}_{\partial D}(X(y)) d|\Phi|_y.
		\end{align} 
		\item[(iii)] $\textbf{n}(s) \in \mathcal{N}_{X(s)} \in \partial D$.
	\end{itemize}
\end{defn}
Note that the Definition \ref{def} ensures that $X_t$ entering \eqref{main-skorohod} is a reflecting process on $\bar{D}$.

\section{Well-posedness of Skorohod-like system}\label{main}

In this section, we establish the well-posedness of the Skorohod-like system by showing the existence, uniqueness and stability of solutions in the sense of Definition \ref{def} to the problem \eqref{main-skorohod}--\eqref{initial-cond}. 

\subsection{Statement of the main results}

The main results of this paper are stated in Theorem \ref{main-theo1}, Theorem \ref{main-theo2} and Theorem \ref{main-stabi}. In the frame of this paper, the focus lies on ensuring the well-posedness of Skorohod solutions to our crowd dynamics problem.

\begin{theorem}[Existence]\label{main-theo1}
	Assume that $(\text{A}_1)$-$(\text{A}_7)$ hold. There exists at least a strong solution to the Skorohod-like system \eqref{main-skorohod}--\eqref{initial-cond} in the sense of Definition \ref{def}.
\end{theorem}
\begin{theorem}[Uniqueness]\label{main-theo2} Assume that $(\text{A}_1)$-$(\text{A}_8)$ hold.
	There is a unique strong solution to \eqref{main-skorohod}--\eqref{initial-cond}.
\end{theorem}
\begin{theorem}[Dependence on parameters]\label{main-stabi}
	Assume that $(\text{A}_1)$-$(\text{A}_5)$ hold and 
	\begin{align}
	\lim_{k \to \infty} E(|X_0^k - X_0|^2) = 0.
	\end{align}
	Suppose that $X_t^n \in C(\bar{S}; \bar{D}\times \bar{D})$ solves
	\begin{align}
	\begin{cases}
	dX_t^n = b(X_t^n(t)dt + \sigma(X_t^n(t))dB(t) + d\Phi_t^n,\\
	X^n(0) = X_0^n \in \bar{D},
	\end{cases}
	\end{align}
	where $X_0^n \in \bar{D}$ is given.
	Then
	\begin{align}
	\lim_{k \to \infty} E(\max_{0 \leq t \leq T}|X_t^n - X_t|^2) = 0,
	\end{align}
	where $X_t \in C(\bar{S}; \bar{D}\times \bar{D})$ is the unique solution of
	\begin{align}
	\begin{cases}
	dX_t = b(X_t(t)dt + \sigma(X_t(t))dB(t) + d\Phi_t,\\
	X(0) = X_0 \in \bar{D},
	\end{cases}
	\end{align}
\end{theorem}
These statements are proven in the next two subsections.

\subsection{Structure of the proof of Theorem \ref{main-theo1}}\label{proof-theo1}

For convenience, we rephrase the solution to the system \eqref{active-nond4} and \eqref{passive-nond4} in terms of the vector $X_t^n$, $n\in \mathbb{N}$, such that
\begin{align}
X_t^n &:= ( X_{a_i}^n(t),X_{b_k}^n(t))^{T} \text{ for } i \in \{1, \dots, N_A\}, k \in \{1, \dots, N_P\} ,\\ F_1(X_t^n,t) &:= \kappa\Upsilon(S(X_{a_i}^n(t)))\frac{\nabla_{X_{a_i}^n} \phi(X_{a_i}^n(t))}{|\nabla_{X_{a_i}^n} \phi(X_{a_i}^n(t))|}( p_{\max}- p(X_{a_i}^n(t), t) ),\\
F_2(X_t^n,t) &:=  \kappa\sum_{j=1}^{N}\frac{X_{c_j}^n(t) - X_{p_k}^n(t) }{\epsilon + |X_{c_j}^n(t) - X_{p_k}^n(t)|}\omega(|X_{c_j}^n(t) - X_{p_k}^n(t)|, S(X_{p_k}^n(t)),\\
\tilde{\sigma}(X_t^n,t)&:= \kappa\beta(S(X_{p_k}^n(t))).
\end{align}
Furthermore, we set
\begin{align}
b(X_t^n,t):= \begin{bmatrix}
F_1(X_t^n,t)\\F_2(X_t^n,t)
\end{bmatrix} \textrm{ and } 
\sigma(X_t^n,t):=\begin{bmatrix}
\textbf{o}\\ \tilde{\beta}
\end{bmatrix},
\end{align}
with
\begin{align}
\textbf{o}:= \begin{bmatrix}
0 &0\\ 0 & 0
\end{bmatrix} \quad \text{ and } \tilde{\beta}:=\begin{bmatrix}
\tilde{\sigma}_{11} & \tilde{\sigma}_{12} \\ \tilde{\sigma}_{21} & \tilde{\sigma}_{22}
\end{bmatrix},
\end{align}
where $\tilde{\sigma}_{IJ} := (\tilde{\sigma}(X_t^n,t))_{IJ}$ for $1 \leq I, J \leq 2$
and the initial datum is 
\begin{align}
X^n(0):= X_0 := \begin{bmatrix}
X_{a_i,0}\\ X_{b_k,0}
\end{bmatrix}.
\end{align}
We denote by $\{\Phi_t^n\}$ the associated process of $\{X_t^n\}$ as in \eqref{main-skorohod}, viz.
\begin{align}
\Phi_t^n:= \begin{bmatrix}
\Phi_1^n(t))\\\Phi_2^n(t)
\end{bmatrix}, \quad t \in S.
\end{align}

We use the compactness method together with the continuity result of the deterministic case stated in Theorem \ref{deterministic-theo} for proving the existence of solutions to \eqref{main-skorohod}-\eqref{initial-cond}. We follow the arguments  by G. Da Prato and J. Zabczyk  ($2014$) (cf. \cite{Prato14}, Section $8.3$) and a result of F. Flandoli (1995) (cf. \cite{Flandoli95}) for martingale solutions. The starting point of this argument is based on considering a sequence $\{X_t^n\}$ of solutions of the following system of Skorohod-like stochastic differential equations
\begin{align}\label{main-skorohod2}
\begin{cases}
dX_t^n = b(X_t^n(h^n(t))dt + \sigma(X_t^n(h^n(t)))dB(t) + d\Phi_t^n,\\
X^n(0) = X_0 \in \bar{D},
\end{cases}
\end{align}
where $X_0^n \in \bar{D}$ is given, and 
\begin{align}\label{step-func1}
h^n(0)=0,
\end{align}
\begin{align}\label{step-func2}
h^n(t)= (k-1)2^{-n}, \quad (k-1)2^{-n} <t\leq  k 2^{-n}, \quad k=1,2,\ldots,n \text{ and } n\geq 1.
\end{align}
Moreover, by Theorem \ref{deterministic-theo}, we have a unique solution of \eqref{main-skorohod2}. Hence, $X_t^n$ obtained for $0 \leq t \leq k2^{-n}$ and for $k2^{-n} < t \leq (k+1)2^{-n}$ is uniquely determined as solution of the following Skohorod equation: 
\begin{align}
X_t^n = X_t^n(k2^{-n}) + b(X_t^n(k2^{-n}))(t-k2^{-n}) + \sigma(X_t^n(k2^{-n}))\{B(t) - B(k2^{-n})\} + \Phi_t^n.
\end{align}
Let us call
\begin{align}
Y_t^n := X_0 + \int_{0}^tb(X_y^n(h^n(y))dy + \int_0^t \sigma(X_y^n(h^n(y))dB(y). 
\end{align}
Then $X_t^n(t) = (\Gamma Y_t^n)(t)$, we also have
\begin{align}\label{subsystem}
Y_t^n := X_0 + \int_{0}^tb((\Gamma Y_t^n)(h^n(y))dy + \int_0^t \sigma((\Gamma Y_t^n)(h^n(y))dB(y). 
\end{align}
We define the family of laws
\begin{eqnarray}\label{laws}
\left\{\mathcal{P}(Y_t^n); 0\leq t \leq T, n\geq 1\right\}.
\end{eqnarray}
\eqref{laws} is a family of probability distributions of $Y_t^n$. Let $\mathcal{P}^n$ be the laws of $Y_t^n$.

The compactness argument proceeds as follows. We begin with $Y_t^n, n \in \mathbb{N}$, given  cf. \eqref{subsystem}. The construction of $Y_t^n$ is investigated on a probability space $(\Omega,\mathcal{F}, P)$ with a filtration $\{\mathcal{F}_t\}$ and a Brownian motion $B(t)$. Next, let $\mathcal{P}^n$ be the laws of $Y_t^n$ which is defined cf. \eqref{laws}. Then, by using Prokhorov's Theorem (cf. \cite{Billingsley1999}, Theorem $5.1$), we can show that the sequence of laws $\{\mathcal{P}^n(Y_t^n)\}$ is weakly convergent as $n \to \infty$ to $\mathcal{P}(Y_t)$ in $C(\bar{S};\mathbb{R}^2 \times \mathbb{R}^2)$. Then, by using the \textquotedblleft Skorohod Representation Theorem\textquotedblright (cf. \cite{Prato14}, Theorem $2.4$), this weak convergence holds in a new probability space with a new stochastic process, for a new filtration. This leads to some arguments for weak convergence results of two stochastic processes in two different probability spaces together with the continuity result in Theorem \ref{deterministic-theo} that we need to use to obtain the existence of our Skorohod-like system \eqref{main-skorohod}. Finally, we prove the uniqueness of solutions to our system. 

\subsection{Proof of Theorem \ref{main-theo1}}

Let us start with handling the tightness of the laws $\{\mathcal{P}^n\}$ through the following Lemma.
\begin{lemma}\label{tightness}
	Assume that $(\text{A}_1)$-$(\text{A}_5)$ hold. Then, the family $\{\mathcal{P}^n\}$ given by \eqref{laws} is tight in $C(\bar{S}, \mathbb{R}^2\times \mathbb{R}^2 )$.
\end{lemma}
\begin{proof}
	To prove the wanted tightness, let us recall the following compact set in $C(\bar{S}, \mathbb{R}^2 \times \mathbb{R}^2)$
	\begin{align}
	K_{M_1M_2} = \left\{f \in C(\bar{S}; \mathbb{R}^2 \times \mathbb{R}^2): \|f\|_{L^\infty(S; \mathbb{R}^2 \times \mathbb{R}^2)} \leq M_1, [f]_{C^\alpha(\bar{S};\mathbb{R}^2 \times \mathbb{R}^2)} \leq M_2\right\}.
	\end{align}
	Now, we show that for a given $\varepsilon > 0$, there are $M_1, M_2 > 0$ such that
	\begin{align}
	P(Y_{\cdot}^n \in K_{M_1M_2}) \leq \varepsilon, \text{ for all } n \in \mathbb{N}.
	\end{align}
	This means that
	\begin{align}
	P(\|Y_{t}^n\|_{L^\infty(S;\mathbb{R}^2 \times \mathbb{R}^2)} > M_1 \text{ or } [Y_{\cdot}^n]_{C^\alpha(\bar{S};\mathbb{R}^2 \times \mathbb{R}^2)} > M_2) \leq \varepsilon.
	\end{align}
	A sufficient condition for this to happen is
	\begin{align}
	P(\|Y_t^n\|_{L^\infty(S; \mathbb{R}^2\times \mathbb{R}^2)} > M_1) < \frac{\varepsilon}{2} \text{ and } P([Y_\cdot^n]_{C^\alpha(\bar{S}; \mathbb{R}^2\times \mathbb{R}^2)} > M_2) < \frac{\varepsilon}{2},
	\end{align}
	where $Y_{\cdot}$ denotes either $Y_t$ or $Y_r$.
	
	We consider first $	P(\|Y_\cdot^n\|_{L^\infty(S, \mathbb{R}^2\times \mathbb{R}^2)} > M_1) < \frac{\varepsilon}{2} $. Using Markov's inequality (see e.g. \cite{Jacod2004}, Corollary 5.1), we get
	\begin{align}
	P(\|Y_t^n\|_{L^\infty(S; \mathbb{R}^2\times \mathbb{R}^2)} > M_1) \leq \frac{1}{M_1}E[\sup_{t\in S} |Y_t^n|],
	\end{align}
	but 
	\begin{align}
	\sup_{t \in S}|Y_t^n| &= \sup_{t \in S} \Bigg\{\left|X_{a_i,0} + \int_{0}^t F_1((\Gamma Y_t^n)(h^n(y)))dy\right|\nonumber\\&, \left|X_{p_k,0} + \int_{0}^t F_2((\Gamma Y_t^n)(h^n(y)))dy + \int_0^t\sigma((\Gamma Y_t^n)(h^n(y)))dB(y)\right|\Bigg\}.
	\end{align}
	We estimate 
	\begin{align}\label{sup1}
	\sup_{t \in S}|Y_t^n| &= \sup_{t \in S} \Bigg\{\left|X_{a_i,0}\right| + \left|\int_{0}^t F_1((\Gamma Y_t^n)(h^n(y)))dy\right|\nonumber\\&, \left|X_{p_k,0}\right| + \left|\int_{0}^t F_2((\Gamma Y_t^n)(h^n(y)))dy\right| + \left|\int_0^t\sigma((\Gamma Y_t^n)(h^n(y)))dB(y)\right|\Bigg\}.
	\end{align}
	Since $F_1, F_2$ are bounded, then we have 
	\begin{align}
	\int_{0}^T F_1((\Gamma Y_t^n)(h^n(y)))dy \leq C \text{ and } \int_{0}^T F_2((\Gamma Y_t^n)(h^n(y)))dy \leq C.
	\end{align}
	Taking the expectation on \eqref{sup1}, we are led to 
	\begin{align}
	E\left[ \sup_{t\in S} |Y_t^n|\right] \leq C + E\left[\sup_{t \in S}\left|\int_0^t \sigma((\Gamma Y_t^n)(h^n(y)))dB(y)\right|\right].
	\end{align}
	On the other hand, the Burlkholder-Davis-Gundy's inequality \footnote[1]{See e.g. \cite{Karatzars2000}, Theorem 3.28 (The Burlkholder-Davis-Gundy's inequality). Let $M \in \mathcal{M }^{c,\text{loc}}$ and call $M_t^{*} := \max_{0\leq s\leq t}|M_s|$. For every $m > 0$, there exists universal positive constants $k_m$, $K_m$ (depending only on $m$), such that the inequalities $$k_m E(\langle M\rangle_T^m \leq E[(M_T^{*})^{2m}] \leq K_mE(\langle M\rangle_T^m)$$ hold for every stopping time $T$. Note that $\mathcal{M }^{c,\text{loc}}$ denotes the space of continuous local martingales and $\langle X\rangle$ represents the quadratic variance process of $X \in \mathcal{M }^{c,\text{loc}}$.}  implies
	\begin{align}
	E\left[\sup_{t \in S}\left|\int_0^t \sigma((\Gamma Y_t^n)(h^n(y)))dB(y)\right|\right]  \leq E\left[\int_0^t |\sigma((\Gamma Y_t^n)(h^n(y)))|^2dy\right]^{1/2}. 
	\end{align}
	Then, we have the following estimate
	\begin{align}
	E\left[ \sup_{t\in S} |Y_t^n|\right]  \leq C+ E\left[\int_0^t |\sigma((\Gamma Y_t^n)(h^n(y)))|^2dy\right]^{1/2} \leq C
	\end{align}
	Hence, for $\varepsilon > 0$, we can choose $M_1>0$ such that $	P(\|Y_t^n\|_{L^\infty(S; \mathbb{R}^2\times \mathbb{R}^2)} > M_1) < \frac{\varepsilon}{2} $.
	
	In the sequel, we consider the second inequality $P([Y_\cdot^n]_{C^\alpha(\bar{S}; \mathbb{R}^2\times \mathbb{R}^2)} > M_2) < \frac{\varepsilon}{2}$, this reads
	\begin{align}
	P([Y_\cdot^n]_{C^\alpha(\bar{S};\mathbb{R}^2 \times \mathbb{R}^2)} > M_2) = P\left(\sup_{t\neq r; t,r \in S} \frac{|Y_t^n - Y_r^n|}{|t-r|^{\alpha}} > M_2\right) \leq \frac{\varepsilon}{2}.
	\end{align}
	
	Let us introduce another class of compact sets now in the Sobolev space $W^{\alpha, p}(0,T; \mathbb{R}^2\times \mathbb{R}^2)$ (which for suitable exponents $\alpha p - \gamma > 1$ lies in $C^\gamma(\bar{S}; \mathbb{R}^2\times \mathbb{R}^2)$). Additionally, we recall the relatively compact sets $K'_{M_1M_2}$, defined as in Section \ref{pre-assumptions}, such that
	\begin{align}
	K'_{M_1M_2} = \left\{ f \in C(\bar{S};\mathbb{R}^2 \times \mathbb{R}^2): \|f\|_{L^{\infty}(S; \mathbb{R}^2 \times \mathbb{R}^2)} \leq M_1, [f]_{W^{\alpha,p}(S; \mathbb{R}^2 \times \mathbb{R}^2)} \leq M_2\right\}.
	\end{align}
	A sufficient condition for $K'_{M_1M_2}$ to be a relative compact underlying space is $\alpha p>1$ (see e.g. \cite{Flandoli95}, \cite{Colangeli2019}). Having this in mind, we wish to prove that there exits $\alpha \in (0,1)$ and $p>1$ with $\alpha p>1$ together with the property: given $\varepsilon > 0$, there is $M_2>0$ such that 
	\begin{align}
	P([Y_{\cdot}^n]_{W^{\alpha, p}(S; \mathbb{R}^2 \times \mathbb{R}^2)} > M_2) < \frac{\varepsilon}{2} \text{ for every } n \in \mathbb{N}.
	\end{align}
	Using Markov's inequality, we obtain 
	\begin{align}
	P([Y_{\cdot}^n]_{W^{\alpha, p}(S; \mathbb{R}^2 \times \mathbb{R}^2)} > M_2) &\leq \frac{1}{M_2} E\left[\int_0^T \int_0^T \frac{|Y_t^n - Y_r^n|^p}{|t-r|^{1+\alpha p}} dt dr\right] \nonumber\\
	&= \frac{C}{M_2} \int_0^T \int_0^T \frac{E[|Y_t^n - Y_r^n|^p]}{|t-r|^{1+\alpha p}} dt dr.
	\end{align}
	For $t>r$, we have
	\begin{align}
	Y_t^n - Y_r^n = \begin{bmatrix}
	\int_r^t F_1((\Gamma Y_t^n)(h^n(y)))dy\\\int_r^t F_2((\Gamma Y_t^n)(h^n(y)))dy
	\end{bmatrix} +
	\begin{bmatrix}
	0\\\int_r^t \sigma (X_y^n(h^n(y))) dB(y)
	\end{bmatrix}.
	\end{align}
	Let us introduce some further notation. For a vector $u=(u_1,u_2)$, we set $|u|:= |u_1|+|u_2|$. At this moment, we consider the following expression 
	\begin{align}
	|Y_t^n - Y_r^n| &= \left| \int_{r}^t F_1((\Gamma Y_t^n)(h^n(y)))dy\right| \nonumber\\&+ \left| \int_{r}^t F_2((\Gamma Y_t^n)(h^n(y)))dy+\int_r^t \sigma((\Gamma Y_t^n)(h^n(y)))dB(y)\right|.
	\end{align}
	Taking the modulus up to the power $p>1$ together with applying Minkowski inequality, we have 
	\begin{align}\label{modulus2}
	|Y_t^n - Y_r^n|^p &= \Bigg(\left| \int_{r}^t F_1((\Gamma Y_t^n)(h^n(y)))dy\right| \nonumber \\&+ \left| \int_{r}^t F_2((\Gamma Y_t^n)(h^n(y)))dy + \int_r^t \sigma((\Gamma Y_t^n)(h^n(y)))dB(y)\right|\Bigg)^p\nonumber \\
	&\leq C\Bigg(\left| \int_{r}^t F_1((\Gamma Y_t^n)(h^n(y)))dy\right|^p \nonumber \\&+ \left| \int_{r}^t F_2((\Gamma Y_t^n)(h^n(y)))dy\right|^p+\left|\int_r^t \sigma((\Gamma Y_t^n)(h^n(y)))dB(y)\right|^p \Bigg) \nonumber\\
	&\leq  C\Bigg(\int_{r}^t |F_1((\Gamma Y_t^n)(h^n(y)))|^pdy +  \int_{r}^t |F_2((\Gamma Y_t^n)(h^n(y)))|^pdy\nonumber \\&+\left|\int_r^t \sigma((\Gamma Y_t^n)(h^n(y)))dB(y)\right|^p\Bigg).
	\end{align}
	Taking the expectation on \eqref{modulus2}, we obtain the following estimate
	\begin{align}\label{e1}
	E[|Y_t^n - Y_r^n|^p] \leq C(t-r)^p + CE\left[ \left|\int_r^t \sigma((\Gamma Y_t^n)(h^n(y)))dB(y)\right|^p\right].
	\end{align}
	Applying the Burkholder-Davis-Gundy's inequality to the second term of the right hand side of \eqref{e1}, we obtain
	\begin{align}
	E\left[ \left|\int_r^t \sigma((\Gamma Y_t^n)(h^n(y)))dB(y)\right|^p\right] \leq C E\left[ \left(\int_r^t dy \right)^{p/2}\right] \leq C(t-r)^{p/2}.
	\end{align}
	On the other hand, if $\alpha < \frac{1}{2}$, then
	\begin{align}
	\int_0^T \int_0^T \frac{1}{|t-r|^{1+(\alpha - \frac{1}{2}) p}} dt dr < \infty.
	\end{align}
	Consequently, we can pick $\alpha < \frac{1}{2}$. Taking now $p>2$ together with the constraint $\alpha p >1$, we can find $M_2 > 0$ such that 
	\begin{align}
	P([Y_t^n]_{W^{\alpha, p}(S;\mathbb{R}^2 \times \mathbb{R}^2)} > M_2)  < \frac{\varepsilon}{2}.
	\end{align}
	This argument completes the proof of the Lemma.
\end{proof}

From Lemma \ref{tightness}, we have obtained that the sequence $\{\mathcal{P}^n\}$ is tight in $C(\bar{S};\mathbb{R}^2 \times \mathbb{R}^2)$. Applying the Prokhorov's Theorem, there are subsequences $\{\mathcal{P}^{n_k}\}$ which converge weakly to some $\mathcal{P}(Y_t)$ as $n \to \infty$. For simplicity of the notation, we denote these subsequences by $\{\mathcal{P}^{n}\}$. This means that we have $\{\mathcal{P}^n\}$ converging weakly to some probability measure $\mathcal{P}$ on Borel sets in $C(\bar{S};\mathbb{R}^2 \times \mathbb{R}^2)$. 

Since we have that $\mathcal{P}^n(Y_t^n)$ converges weakly to $\mathcal{P}(Y_t)$ as $n \to \infty$, by using the \textquotedblleft Skorohod Representation Theorem\textquotedblright, there exists a probability space $(\widetilde{\Omega}, \tilde{\mathcal{F}},\tilde{P})$ with the filtration $\{\tilde{\mathcal{F}}_t\}$ and $\tilde{Y}_t^n$, $\tilde{Y}_t$ belonging to $C(\bar{S}; \mathbb{R}^2 \times \mathbb{R}^2)$ with $n\in \mathbb{N}$, such that $\mathcal{P}(\tilde{Y}) = \mathcal{P}(Y)$, $\mathcal{P}(\tilde{Y}_t^n) = \mathcal{P}(Y_t^n)$, and $\tilde{Y}_t^n \to \tilde{Y}_t$ as $n \to \infty$, $\tilde{P}-$a.s.
Moreover, let $(\tilde{X}_t^n, \tilde{\Phi}_t^n)$ and $(\tilde{X}_t, \tilde{\Phi}_t)$ be the solutions of the Skorohod equations
\begin{align}\label{sols}
\tilde{X}_t^n &= \tilde{Y}_t^n + \tilde{\Phi}_t^n, \nonumber\\
\tilde{X}_t &= \tilde{Y}_t + \tilde{\Phi}_t,
\end{align}
respectively.  Then the continuity result in Theorem \ref{deterministic-theo} implies that the sequence $(\tilde{X}_t^n, \tilde{\Phi}_t^n)$ converges to $(\tilde{X}_t, \tilde{\Phi}_t) \in C(\bar{S}; \bar{D}\times \bar{D}) \times C(\bar{S})$ uniformly in $t\in \bar{S}$, $\tilde{P}-$a.s as $n \rightarrow \infty$. Hence, we  still need to prove that $\tilde{Y}_t^n$ converges to $\tilde{Y}_t$ in some sense, where we denote
\begin{align}
\tilde{Y}_t^n := \tilde{X}_0 + \int_{0}^tb(\tilde{X}_y^n(h_n(y))dy + \int_0^t \sigma(\tilde{X}_y^n(h_n(y))d\tilde{B}(y).
\end{align}
and 
\begin{align}
\tilde{Y}_t := \tilde{X}_0 + \int_{0}^tb(\tilde{X}_y^n((y))dy + \int_0^t \sigma(\tilde{X}_y^n((y))d\tilde{B}(y).
\end{align}
To complete the proof of the existence of solutions to the problem \eqref{main-skorohod}-\eqref{initial-cond} in the sense of Definition \ref{def}, we consider the following Lemma.
\begin{lemma}\label{solu-theo}
	The pair $(\tilde{X}_t, \tilde{\Phi}_t) \in C(\bar{S}; \bar{D}\times \bar{D}) \times C(\bar{S})$ cf. \eqref{sols} is a solution of the Skorohod-like system
	\begin{align}
	\tilde{X}_t = \tilde{X}_0 + \int_{0}^tb(\tilde{X}_y(y))dy + \int_0^t \sigma(\tilde{X}_y(y))d\tilde{B}(y) + \tilde{\Phi}_t
	\end{align}
	with $\tilde{X}_0 \in \bar{D}$.
\end{lemma}
\begin{proof}
	We consider the term $\int_0^t \sigma(\tilde{X}_t^n(h_n(y))d\tilde{B}(y)$ with the step process $\sigma(\tilde{X}_t^n(h_n(y))$. Approximating this stochastic integral by Riemann-Stieltjes sums (see e.g. \cite{Evans2013}), it yields 
	\begin{align}\label{reimann-app}
	\int_{0}^t\sigma(\tilde{X}_y^n(h_n(y))d\tilde{B}(y) = \sum_{k=0}^{n-1} \sigma(\tilde{X}_t^n(h_n(t)))(B(t_{k+1}^n) - B(t_k^n)).
	\end{align}
	This gives by taking the limit $n \to \infty$ in \eqref{reimann-app}
	\begin{align}\label{reimann}
	\lim_{n\to \infty}\int_{0}^t\sigma(\tilde{X}_y^n(h_n(y))d\tilde{B}(y) = \lim_{n\to \infty}\sum_{k=0}^{n-1} \sigma(\tilde{X}_t^n(h_n(t)))(B(t_{k+1}^n) - B(t_k^n)) \nonumber\\=\sum_{k=0}^{n-1} \sigma(\tilde{X}_t(t))(B(t_{k+1}) - B(t_k))  = \int_0^t \sigma(\tilde{X}_y(y))d\tilde{B}(y).
	\end{align}
	By the fact that $(\tilde{X}_t^n, \tilde{\Phi}_t^n)$ converges to $(\tilde{X}_t, \tilde{\Phi}_t)\in C(\bar{S}; \bar{D} \times \bar{D})\times C(\bar{S})$ uniformly in $t\in [0,T]$ $\tilde{P}-$a.s as $n \rightarrow \infty$ together with \eqref{reimann}, we obtain that
	\begin{align}
	\tilde{X}_t^n = \tilde{X}_0 + \int_{0}^tb(\tilde{X}_y^n(h^n(y)))dy + \int_0^t \sigma(\tilde{X}_y^n(h^n(y)))d\tilde{B}(y) + \tilde{\Phi}_t^n.
	\end{align}
	converges to 
	\begin{align}
	\tilde{X}_t = \tilde{X}_0 + \int_{0}^tb(\tilde{X}_y(y))dy + \int_0^t \sigma(\tilde{X}_y(y))d\tilde{B}(y) + \tilde{\Phi}_t, \quad \tilde{P}-\text{a.s as } n\to \infty.
	\end{align}
\end{proof}	

\subsubsection{Proof of Theorem \ref{main-theo2}}

\begin{proof}
	We take $X_t, X'_t \in C(\bar{S}; \bar{D} \times \bar{D})$ two solutions to \eqref{main-skorohod}-\eqref{initial-cond} with the same initial values $X(0) = X'(0)$. 
	
	Moreover, suppose that the supports of $b$ and $\sigma$ are included in the same ball $B(x_0,\delta)$ for some $x_0 \in \partial D$. We use the proof idea of Lemma $5.3$ in \cite{Saisho1987}. Let us recall the assumption $(\text{A}_8)$, where $D$ satisfies the following condition: It exists a positive number $\nu$ such that for each $x_0 \in \partial D$ we can find $f\in C^2(\mathbb{R}^2 \times \mathbb{R}^2 )$ satisfying $$\langle y-x, \textbf{n}\rangle + \frac{1}{\nu}\langle\nabla f(x), \textbf{n}\rangle|y-x|^2 \geq 0.$$
	for any $x \in B(x_0,\delta') \cap \partial D, y \in B(x_0, \delta'') \cap \partial \bar{D}$ and $\textbf{n} \in \mathcal{N}_{x}$. Then, we have 
	\begin{align}\label{cond-1}
	\langle X_s - X_s', d\Phi_s - d\Phi'_s\rangle - \frac{1}{\nu}|X_s - X'_s|^2\langle\textbf{l}, d\Phi_s - d\Phi'_s\rangle \nonumber\\
	=-(\langle X_s - X_s', d\Phi_s\rangle +  \frac{1}{\nu}|X_s - X'_s|^2\langle\textbf{l}, d\Phi_s\rangle) \nonumber\\- (\langle X_s - X_s', d\Phi'_s\rangle +  \frac{1}{\nu}|X_s - X'_s|^2\langle\textbf{l}, d\Phi'_s\rangle) \leq 0,
	\end{align}
	where $\textbf{l}$ is the unit vector appearing in Condition $(\text{A}_7)$.
	
	Using similar ideas as in \cite{Lions1984}, Proposition $4.1$, we have the following estimate
	\begin{align}\label{est-uni1}
	&|X_t - X'_t|^2\exp\left\{-\frac{1}{\nu}(\Phi(X_t) - \Phi'(X_t))\right\} \leq \nonumber\\&2\Bigg(\exp\left\{-\frac{1}{\nu}(\Phi(X_y) - \Phi'(X_y))\right\}\int_0^t(b(X_y(y)) - b(X'_y(y)))dy \nonumber\\&+ \exp\left\{-\frac{1}{\nu}(\Phi(X_y) - \Phi'(X_y))\right\}\int_0^t(\sigma(X_y(y)) - \sigma(X_y(y))) dB(y) \Bigg)^2 \nonumber\\
	&+ \exp\left\{-\frac{1}{\nu}(\Phi(X_y) - \Phi'(X_y))\right\}\int_0^t\left(2\langle X_y - X'_y,l\rangle - \frac{1}{\nu}|X_y - X'_y|^2\right)d\Phi_y\nonumber\\&+\exp\left\{-\frac{1}{\nu}(\Phi(X_y) - \Phi'(X_y))\right\}\int_0^t\left(2\langle X_y - X'_y,l\rangle - \frac{1}{\nu}|X_y - X'_y|^2\right)d\Phi'_y \nonumber\\&2\int_0^t\left|b(X_y(y)) - b(X'_y(y))\right|^2\exp\left\{-\frac{2}{\nu}(\Phi(X_y) - \Phi'(X_y))\right\}dy \nonumber\\&+ 2\int_0^t\left|\sigma(X_y(y)) - \sigma(X_y(y))\right|^2\exp\left\{-\frac{2}{\nu}(\Phi(X_y) - \Phi'(X_y))\right\} dy \nonumber\\
	&+ \int_0^t\left(2\langle X_y - X'_y,l\rangle - \frac{1}{\nu}|X_y - X'_y|^2\right)\exp\left\{-\frac{1}{\nu}(\Phi(X_y) - \Phi'(X_y))\right\}d\Phi_y\nonumber\\&+\int_0^t\left(2\langle X_y - X'_y,l\rangle - \frac{1}{\nu}|X_y - X'_y|^2\right)\exp\left\{-\frac{1}{\nu}(\Phi(X_y) - \Phi'(X_y))\right\}d\Phi'_y.
	\end{align}
	On the other hand, taking the expectation are both sides of \eqref{est-uni1} and using the Lipschitz condidion to the first term of the right hand side together with \eqref{cond-1},  we are led to
	
	\begin{align}
	E\left(|X_t - X'_t|^2\exp\left\{-\frac{1}{\nu}(\Phi(X_t) - \Phi'(X_t))\right\}\right) &\leq \nonumber\\ C\int_0^tE\Big(|X_y &- X'_y|^2\exp\left\{-\frac{2}{\nu}(\Phi(X_y) - \Phi'(X_y))\right\}\Big)dy.
	\end{align}
	This also implies that
	
	\begin{align}
	E[|X_t - X_t'|^2] \leq C\int_0^tE[|X_y - X'_y|^2]dy.
	\end{align}
	Hence, $X_t=X_t'$ for all $t \in [0,T]$. Then, the pathwise uniqueness of solutions to \eqref{main-skorohod} holds true. On the other hand, combining the Lemma \ref{solu-theo} together with the fact that the pathwise uniqueness implies the uniqueness of strong solutions (see in \cite{Ikeda1981}, Theorem IV-1.1). Therefore, there is a unique solution \\$(X_t, \Phi_t) \in C(\bar{S}; \bar{D}\times \bar{D})\times C(\bar{S})$ of \eqref{main-skorohod}.
\end{proof}


\subsection{Proof of Theorem \ref{main-stabi}}

\begin{proof}
	Let us recall our system of SDEs
	
	\begin{align}
	\begin{cases}
	dX_t^n = b(X_t^n(t)) dt + \sigma (X_t^n(t))dB(t) + d\Phi_t^n,\\
	X^n(0) = X_0^n \in \bar{D} \text{ for } n \geq 1.
	\end{cases}
	\end{align}
	Then, we have
	
	
	\begin{align}\label{sta-1}
	X_t^n = X_0^n + \int_0^t b(X_z^n(z))dz + \int_0^t  \sigma (X_z^n(z))dB(z) + \Phi_t^n,
	\end{align}
	Let us consider the following equation
	
	\begin{align}
	X_t^n - X_t &= X_0^n - X_0 + \int_0^t b(X_z^n(z))dz - \int_0^t b(X_z(z))dz \nonumber\\&+ \int_0^t \sigma (X_z^n(z))dB(z) - \int_0^t \sigma (X_z(z))dB(z) + \Phi_t^n - \Phi_t.
	\end{align}
	Since
	$(a + b+ c + d)^2 \leq 4a^2 + 4b^2+ 4c^2+ 4d^2$ for any $a, b, c, d \in \mathbb{R}$, we have the following estimate
	\begin{align}\label{sta-2}
	|X_t^n - X_t|^2 &\leq 4|X_0^n - X_0|^2 + 4\left|\int_0^t b(X_z^n(z))dz - \int_0^t b(X_z(z))dz\right|^2 \nonumber\\&+ 4\left|\int_0^t \sigma (X_z^n(z))dB(z) - \int_0^t \sigma (X_z(z))dB(z)\right|^2 + 4|\Phi_t^n - \Phi_t|^2.
	\end{align} 
	Taking the expectation on both sides of  \eqref{sta-2}, we have
	\begin{align}\label{stabi-1}
	E(|X_t^n - X_t|^2) &\leq 4E(|X_0^n - X_0|^2) + 4E\left(\left|\int_0^t b(X_z^n(z))dz - \int_0^t b(X_z(z))dz\right|^2\right) \nonumber\\&+ 4E\left(\left|\int_0^t \sigma (X_z^n(z))dB(z) - \int_0^t \sigma (X_z(z))dB(z)\right|^2\right) + 4E\left(|\Phi_t^n - \Phi_t|^2\right).
	\end{align} 
	To begin with, we consider the second term and the third term of the right-hand side of \eqref{stabi-1}. Using Cauchy-Schwarz inequality together with the assumption that $b, \sigma$ are Lipschitz functions, we are led to
	
	\begin{align}\label{sta-1-1}
	E\left(\left|\int_0^t b(X_z^n(z))dz - \int_0^t b(X_z(z))dz\right|^2\right) &\leq CE\left(\int_0^t |b(X_z^n(z)) - b(X_z(z))|^2dz\right)\nonumber\\
	&\leq C\int_0^tE(|X_z^n - X_z|^2)dz
	\end{align}
	and 
	
	\begin{align}\label{sta-1-2}
	E\left(\left|\int_0^t \sigma (X_z^n(z))dB(z) - \int_0^t \sigma (X_z(z))dB(z)\right|^2\right) = E\left(\int_0^t |\sigma (X_z^n(z)) - \sigma (X_z(z))|^2dz\right)\nonumber\\
	\leq C\int_0^t E(|X_z^n - X_z|^2)dz.
	\end{align}
	Moreover, using \eqref{reflection-func}, it yields
	\begin{align}
	|\Phi_t^n - \Phi_t| &\leq 2|X_0^n - X_0| + 2\left|\int_0^t b(X_z^n(z))dz - \int_0^t b(X_z(z))dz\right| \nonumber\\&+ 2\left|\int_0^t \sigma (X_z^n(z))dB(z) - \int_0^t \sigma (X_z(z))dB(z)\right|.
	\end{align}
	Since $(a+b+c)^2 \leq 3a^2+3b^2+3c^2$ for all $a, b, c \in \mathbb{R}$, then we have
	\begin{align}\label{sta-3}
	|\Phi_t^n - \Phi_t|^2 &\leq 6|X_0^n - X_0|^2 + 6\left|\int_0^t b(X_z^n(z)) - \int_0^t b(X_z(z))dz\right|^2 \nonumber\\&+ 6\left|\int_0^t \sigma (X_z^n(z))dB(z) - \int_0^t \sigma (X_z(z))dB(z)\right|^2.
	\end{align}
	Taking the expectation on both sides of \eqref{sta-3}, we are led to
	\begin{align}\label{stabi-es1}
	E(|\Phi_t^n - \Phi_t|^2) &\leq 6E(|X_0^n - X_0|^2) + 6E\left(\left|\int_0^t b(X_z^n(z)) - \int_0^t b(X_z(z))dz\right|^2\right) \nonumber\\&+ 6E\left(\left|\int_0^t \sigma (X_z^n(z))dB(z) - \int_0^t \sigma (X_z(z))dB(z)\right|^2\right).
	\end{align}
	Apply Cauchy-Schwarz's inequality to the second and third terms of the right-hand side of \eqref{stabi-es1}, we have the following estimate
	\begin{align}
	E(|\Phi_t^n - \Phi_t|^2) &\leq 6E(|X_0^n - X_0|^2) + 
	CE\left(\int_0^t |b(X_z^n(z))dz - b(X_z(z))|^2dz\right) 
	\nonumber\\&
	+ 6E\left(\int_0^t \left|\sigma(X_z^n(z))dB(z) - \sigma(X_z(z))\right|^2dz\right).
	\end{align}
	Using again the assumption that $b, \sigma$ are Lipschitz functions, we get the following estimate
	\begin{align}\label{sta-phi}
	E(|\Phi_t^n - \Phi_t|^2) \leq 6E(|X_0^n - X_0|^2) + C\int_0^t E(|X_z^n(z) - X_z(z)|^2)dz \nonumber\\+ C\int_0^t E(|X_z^n(z) - X_z(z)|^2)dz.
	\end{align}
	Using \eqref{sta-1-1}, \eqref{sta-1-2} and \eqref{sta-phi}, then the inequality \eqref{sta-1} becomes
	\begin{align}\label{bfgronwall}
	E(|X_t^n - X_t|^2) \leq CE(|X_0^n - X_0|^2) + C\int_0^t E(|X_z^n(z) - X_z(z)|^2)dz,
	\end{align}
	for $0 \leq t \leq T$.
	Applying Gronwall?s inequality to \eqref{bfgronwall}, it yields
	\begin{align}\label{gronwall}
	E(|X_t^n - X_t|^2) \leq CE(|X_0^n - X_0|^2).
	\end{align}
	Moreover, we have that
	
	\begin{align}\label{sta-4}
	\max_{0\leq t \leq T}|X_t^n - X_t|^2 &\leq 4|X_0^n - X_0|^2 + C\int_0^T |X_t^n(t) - X_t(t)|^2dt \nonumber\\&+ 4\max_{0 \leq t \leq T}\left|\int_0^T \sigma (X_t^n(t)) - \sigma (X_t(t))dB(t)\right|^2 + 4\max_{0 \leq t \leq T} |\Phi_t^n - \Phi_t|^2.
	\end{align}
	After taking the expectation on both sides of \eqref{sta-4}, we apply the martingale inequality to the third term on the right-hand side of the resulting inequality, which reads
	\begin{align}
	E\left(\max_{0\leq t \leq T}|X_t^n - X_t|^2\right) &\leq 4E(|X_0^n - X_0|^2) + C\int_0^T E(|X_t^n(t) - X_t(t)|^2)dt \nonumber\\&+ C\int_0^T E(|X_t^n(t) - X_t(t)|^2)dt + 4E\left(\max_{0 \leq t \leq T} |\Phi_t^n - \Phi_t|^2\right) \nonumber\\
	&\leq CE(|X_0^n - X_0|^2) + C\int_0^T E(|X_t^n(t) - X_t(t)|^2)dt. 
	\end{align}
	Finally, using \eqref{sta-phi} and \eqref{gronwall}, we obtain the desired
	\begin{align}
	E\left(\max_{0\leq t \leq T}|X_t^n - X_t|^2\right) \leq CE(|X_0^n - X_0|^2).
	\end{align}
	By the fact that 
	$\lim_{n\to \infty} E(|X_0^n - X_0|^2) = 0$,
	we obtain the following estimate
	\begin{align}
	\lim_{n \to \infty} E\left(\max_{0\leq t \leq T}|X_t^n - X_t|^2\right) = 0.
	\end{align}
\end{proof}
\section{Concluding remarks}\label{conclusion}

In this paper, we have shown the existence and uniqueness of solutions to a system of Skorohod-like stochastic differential equations modeling  the dynamics of a mixed population of active and passive pedestrians walking within a heterogenous environment in the presence of a stationary fire. Due to the discomfort pressure term as well as to the Morse potential preventing particles (pedestrians) to overlap, our model is nonlinearly coupled. The main feature of the model is that the dynamics of the crowd takes place in an heterogeneous domain. i.e. obstacles hinder the motion. Hence, to allow the SDEs to account for the presence of the obstacles, we formulate our crowd dynamics scenario as a Skorohod-like system with reflecting boundary condition posed in a bounded domain in $\mathbb{R}^2$. Then we use compactness methods to prove the existence of solutions.  The uniqueness of solutions follows by standard arguments. 

There are a number of open issues that are worth to be investigated for our system:

1. To obtain a better insight on how the solution of the SDEs behave and how close is this behaviour to what is expected from standard evacuation scenarios, a convergent numerical approximation of solutions to \eqref{main-skorohod}-\eqref{initial-cond} needs to be implemented. One possible route is to design an iterative weak approximation of the Skorohod system as it is done e.g. in \cite{Bossy2004}, \cite{Onskog2010}, and in the references cited therein. The main challenge is to get fast and accurate numerical approximation of solutions so that an efficient parameter identification strategy can be proposed.

2. We did assume that the fire is a stationary obstacle, i.e.  $\partial F$ is independent on $t$. But, even if not evolving, this fire-obstacle should in fact have a time dependent boundary. Using the working technique from \cite{Onskog2010}, we expect that it is possible to handle the case of a time-evolving fire, provided the shape of the fire $\partial F(t)$ is sufficiently regular, fixed in space,  and  {\em a priori} prescribed.

3. From a mathematical point of view, the situation becomes a lot more challenging when there is a feedback mechanism between the pedestrian dynamics and the environment (fire and geometry). Empirically, such pedestrians-environment feedback was pointed out in  \cite{Ronchi2019}.  An extension can be done in this context  using the smoke observable $s(x,t)$. As a further development of our model, we intend to incorporate the "transport" of smoke eventually  via a measure-valued equation (cf. e.g. \cite{Hille2016}), coupled with our SDEs for the pedestrian dynamics.  In this case, besides the well-posedness question, it is interesting to study the large-time behavior of the system of evolution equations. 
Instead of a measure-valued equation for the smoke dynamics, one could also use a stochastically perturbed diffusion-transport equation. In this case, the approach from \cite{Crisan2018} is potentially applicable, provided the coupling between the SDEs for the crowd dynamics and the SPDE for the smoke evolution is done in a well-posed manner. However, in both cases, it is not yet clear cut how to couple correctly the model equations. 

4. From the modeling point of view, it would be very useful to find out to which extent the motion of active particles can affect the motion of passive particles so that the overall evacuation time is reduced. Note that our crowd dynamics context  does not involve leaders, and besides the social pressure and the repelling from overlapping,  there are no other imposed interactions between pedestrians. In this spirit, we are close to the setting described in \cite{Cejkova2018}, where active and passive particles interplay together to find exists in a maze. Further links between maze-solving strategies and our crowd dynamics scenario would need to be identified to make progress in this direction.

5. The model validation is an open question in this context. Hence, a suitable
experiment design is needed to make any progress in this sense. For instance, the experiments
in \cite{Horiuchi1986} can serve as a typical example for the relevance of distinguishing between two groups of occupants: regular users of the building and those less familiar with it.

6. Further extensions of this modeling approach can be foreseen. One particularly interesting direction would be either to endow the particles with "opinions" and let them be open-minded or closed-minded, or to recast the overall setting into a more classical opinion dynamics framework, where two distinct partial opinions compete to reach coherent patterns and collective consensus; see e.g. \cite{Pent2007} for interactions between opinions of a majority vs. those of a minority, fight for social influence  \cite{Pent2007,Groeber2014}, building opinions when limited information is available \cite{Mavrodiev2013}.

\appendix

\section{Regularized Eikonal equation for motion planning}\label{Eikonal}

To describe how the active population moves within $D$, we use a motion planning in terms of the solution of the following regularized Eikonal equation:
\begin{align}\label{viscous-eikonal}
\begin{cases}
-\varsigma\Delta \phi_{\varsigma} + |\nabla\phi_{\varsigma}|^2 = f^2 \quad &\text{ in } D,\\
\phi_{\varsigma} = 0 \quad &\text{ at } E,\\
\nabla\phi_{\varsigma}\cdot \textbf{n} = 0 \quad &\text{ at } \partial (\Lambda \setminus (G \cup E \cup F)),
\end{cases}
\end{align}
where $\varsigma >0$ given sufficiently small. In fact, $|\nabla \phi_{\varsigma}|$ plays the role of {\em a priori} known guidance (navigation information). Inspired very much from the implementation of video games, this is a strategy commonly used in most major crowd evacuation softwares, i.e. the map of the building to be evacuated is built-in. An alternative motion guidance strategy is suggested in \cite{Yang2014}.

We point out the existence and uniqueness of classical solutions to the problem \eqref{viscous-eikonal} in the following Lemma.
\begin{lemma}
	Assume that $f \in C^{\alpha}(D)$ with $0<\alpha <1$. Let $D \subset \mathbb{R}^2$ be a domain with $\partial D \cup \partial G \in C^{2,\alpha}$. Then the problem \eqref{viscous-eikonal} has a unique solution $\phi_{\varsigma} \in C(\overline{D})\cap C^2(D)$. 
\end{lemma}
\begin{proof} The idea of this proof comes from Theorem 2.1, p.10, in \cite{Schieborn:2006} for the case of the Dirichlet problem. 
	In fact, the semilinear viscous problem \eqref{viscous-eikonal} can be transformed into a linear partial differential equation via $w_a: D \longrightarrow \mathbb{R}$ given by
	\begin{align}
	w_a(\phi_\varsigma) := \exp(-\varsigma^{-1}\phi_\varsigma) - 1,
	\end{align}
	where $a = \frac{1}{\varsigma}$.
	Then $w_a \in C(\overline{D})\cap C^2(D)$ becomes a solution of the following linear partial differential equation with mixed Dirichlet-Neumann boundary conditions:
	\begin{align}
	\begin{cases}\label{transform:eqn}
	-\Delta w_a + f^2 a^2 w_a + a^2 = 0 \quad &\text{ in } D,\\
	w_a = 0 \quad &\text{ at } E,\\
	\nabla w_a\cdot \textbf{n} = 0 \quad &\text{ at } \partial D \cup \partial G.
	\end{cases}
	\end{align}
	Futhermore, there is a unique solution $w_a \in C(\overline{D})\cap C^2(D)$ of the problem \eqref{transform:eqn} (see in Theorem 1, \cite{Lieberman1986}). This also implies that there is a unique solution $\phi_\varsigma\in C(\overline{D})\cap C^2(D)$ to the problem \eqref{viscous-eikonal}.
\end{proof}

\section{Nondimensionalization}\label{nondim}

In this section, we nondimensionalize the system \eqref{active}-\eqref{passive}. By this procedure, we aim to identify the relevant characteristic time and length scales involved in this crowd dynamics scenario. We let $\hat{D}$ denote the scaled set $\frac{1}{x_{\text{ref}}} D$. We introduce $x_{a_i}^{\text{ref}}, x_{p_k}^{\text{ref}}, x_{\text{ref}}$ and $t_{a_i}^{\text{ref}}, t_{p_k}^{\text{ref}}, t_{\text{ref}}$ as possible characteristic length and time scales, respectively. We choose

$X_{a_i}(t_\text{ref}\tau):=\frac{x_{a_i}(t)}{x_{a_i}^{\text{ref}}}, X_{p_k}(t_\text{ref}\tau):= \frac{x_{p_k}(t)}{x_{p_k}^{\text{ref}}}, z:=\frac{x}{x_{\text{ref}}}$ and $\tau:= \frac{t_{a_i}}{t_{a_i}^{\text{ref}}} = \frac{t_{p_k}}{t_{p_k}^{\text{ref}}} = \frac{t}{t_{\text{ref}}}$
where $x_{a_i}^{\text{ref}} = x_{p_k}^{\text{ref}}=x_{\text{ref}}$ and  $t_{a_i}^{\text{ref}}=t_{p_k}^{\text{ref}}= t_{\text{ref}}$. Then, equations \eqref{active} and \eqref{passive} become
\begin{align}\label{active-nond1}
\begin{cases}
\frac{x_{\text{ref}}}{t_{\text{ref}}}\frac{d}{d\tau}X_{a_i}(t_{\text{ref}}\tau) &= \Upsilon_{\text{ref}}s_{\text{ref}}\tilde{\Upsilon}(S(x_{\text{ref}}X_{a_i}(t_{\text{ref}}\tau)))\frac{\phi_{\text{ref}}\nabla_{X_{a_i}} \tilde{\phi}(x_{\text{ref}}X_{a_i}(t_{\text{ref}}\tau))}{\phi_{\text{ref}}|\nabla_{X_{a_i}} \tilde{\phi}(x_{\text{ref}}X_{a_i}(t_{\text{ref}}\tau))|}(p_{\max}\\&- p_{\text{ref}}\tilde{p}(x_{\text{ref}}X_{a_i}(t_{\text{ref}}\tau), t_{\text{ref}}\tau) ),\\
X_{a_i}(0) &= \frac{X_{a_{i}0}}{x_{\text{ref}}},
\end{cases}
\end{align}
where
\begin{align}
p(\textbf{x}_{a_i}(t), t) &=p_{\text{ref}}\tilde{p}(x_{\text{ref}}X_{a_i}(t_{\text{ref}}\tau), t_{\text{ref}}\tau) \nonumber\\&= \mu_{\text{ref}}\tilde{\mu}( t_{\text{ref}}\tau_{a_i})\int_{\hat{\Omega}\cap B(x_{\text{ref}}X_{a_i},\tilde{\delta}_{\text{ref}}\hat{\tilde{\delta}})} \sum_{j=1}^N\delta(y_{\text{ref}}Y - x_{\text{ref}}X_{c_j}(t_{\text{ref}}\tau)y_{\text{ref}}dY.
\end{align}
\begin{align}\label{passive-nond1}
\begin{cases}
\frac{x_{\text{ref}}}{t_{\text{ref}}}\frac{d}{d\tau}X_{p_k}(t_{\text{ref}}\tau) &= \sum_{j=1}^{N}\frac{x_{\text{ref}}X_{c_j} - x_{\text{ref}}X_{p_k} }{\epsilon + |x_{\text{ref}}X_{c_j} - x_{\text{ref}}X_{p_k}|}\omega_{\text{ref}}\tilde{\omega}(|x_{\text{ref}}X_{c_j} - x_{\text{ref}}X_{p_k}|, S(x_{\text{ref}}X_{p_k}(t_{\text{ref}}\tau))) \\&+ \beta_{\text{ref}}\tilde{\beta}(S(x_{\text{ref}}X_{p_k}( t_{\text{ref}}\tau)))\sqrt{d\tau}\frac{d\tilde{B}(t_{\text{ref}}\tau)}{\sqrt{d\tau}},\\
X_{p_k}(0) &= \frac{X_{p_k0}}{x_{\text{ref}}},
\end{cases}
\end{align}
where
\begin{align}
\omega(y,z) =\omega_{\text{ref}}\tilde{\omega} (y_{\text{ref}}\tilde{y}, z_{\text{ref}}\tilde{z}) = -\beta_{\text{ref}}\beta(z_{\text{ref}}\tilde{z}) \left(-C_A e^{-\frac{y_{\text{ref}}\tilde{y}}{\ell_A} }+ C_R e^{-\frac{y_{\text{ref}}\tilde{y}}{\ell_R}}\right),
\end{align}
\begin{align}
\beta(y)=\beta_{\text{ref}}\tilde{\beta}(y_{\text{ref}}M)=
\begin{cases}
1, \text{ if } y_{\text{ref}}M < s_{\text{cr}},\\
0, \text{ if } y_{\text{ref}}M \geq s_{\text{cr}}.
\end{cases}
\end{align}
Multiplying \eqref{active-nond1} by $\frac{t_{\text{ref}}}{x_{\text{ref}}}$, we are led to
\begin{align}\label{active-nond2}
\begin{cases}
\frac{d}{d\tau}X_{a_i}(t_{\text{ref}}\tau) &= \frac{\Upsilon_{\text{ref}}t_{\text{ref}}s_{\text{ref}} }{x_{\text{ref}}}\tilde{\Upsilon}(S(x_{\text{ref}}X_{a_i}(t_{\text{ref}}\tau)))\frac{\phi_{\text{ref}}\nabla_{X_{a_i}} \tilde{\phi}(x_{\text{ref}}X_{a_i}(t_{\text{ref}}\tau))}{\phi_{\text{ref}}|x_{\text{ref}}\nabla_{X_{a_i}} \tilde{\phi}(X_{a_i}(t_{\text{ref}}\tau))|}(p_{\max}\\&- p_{\text{ref}}\tilde{p}(x_{\text{ref}}X_{a_i}(t_{\text{ref}}\tau), t_{\text{ref}}\tau) ),\\
X_{a_i}(0) &= \frac{X_{a_{i}0}}{x_{\text{ref}}}.
\end{cases}
\end{align}
Similarly, we obtain
\begin{align}\label{passvive-nond2}
\begin{cases}
\frac{d}{d\tau}X_{p_k}(t_{\text{ref}}\tau) &= \frac{\omega_{\text{ref}} t_{\text{ref}}}{x_{\text{ref}}}\sum_{j=1}^{N}\frac{x_{\text{ref}}X_{c_j} - x_{\text{ref}}X_{p_k} }{\epsilon + |x_{\text{ref}}X_{c_j} - x_{\text{ref}}X_{p_k}|}\tilde{\omega}(|x_{\text{ref}}X_{c_j} - x_{\text{ref}}X_{p_k}|, S(x_{\text{ref}}X_{p_k}(t_{\text{ref}}\tau))) \\&+ \frac{\beta_{\text{ref}}t_{\text{ref}}}{x_{\text{ref}}}\tilde{\beta}(S(x_{\text{ref}}X_{p_k}(t_{\text{ref}}\tau)))\sqrt{d\tau}\frac{d\tilde{B}(t_{\text{ref}}\tau)}{\sqrt{d\tau}},\\
X_{p_k}(0) &= \frac{X_{p_k0}}{x_{\text{ref}}},
\end{cases}
\end{align}
From \eqref{active-nond2} and \eqref{passvive-nond2} the following dimensionless numbers arise:
\begin{align}
\frac{\Upsilon_{\text{ref}}t_{\text{ref}}s_{\text{ref}} p_{\max}}{x_{\text{ref}}}, \quad \frac{\Upsilon_{\text{ref}}t_{\text{ref}}s_{\text{ref}} p_{\text{ref}}}{x_{\text{ref}}},\quad \frac{\omega_{\text{ref}} t_{\text{ref}}}{x_{\text{ref}}}, \quad  \frac{\beta_{\text{ref}}t_{\text{ref}}}{x_{\text{ref}}}.
\end{align}
These dimensionless numbers indicate four different choices of the characteristic time scale $t_{\text{ref}}$. This is due to the complexity of our system: active and passive agents interplay within the domain geometry as well as the propagation of the smoke. The choice of the corresponding time scale can be the characteristic time capturing relation between the smoke extinction, the walking speed and the discomfort level to the overall population size or the local discomfort, the one for the drift from the smoke propagation, the one for the drift produced by the action of active and passive pedestrians and the one for the amplifying factor on the noise. Therefore, in order to cover the physical relevance of the whole system, we introduce the following rate
\begin{align}
\kappa := \max\left\{\frac{\Upsilon_{\text{ref}}t_{\text{ref}}s_{\text{ref}} p_{\max}}{x_{\text{ref}}}, \frac{\Upsilon_{\text{ref}}t_{\text{ref}}s_{\text{ref}} p_{\text{ref}}}{x_{\text{ref}}}, \frac{w_{\text{ref}} t_{\text{ref}}}{x_{\text{ref}}}, \frac{\beta_{\text{ref}}t_{\text{ref}}}{x_{\text{ref}}}\right\}.
\end{align}
On the other hand, a typical choice for the reference length scale is $x_{\text{ref}} = \ell$, where $\ell:= \text{diam}(D)$. 
Finally, we obtain the following nondimensionalized equations
\begin{align}\label{active-nond4}
\begin{cases}
\frac{d}{d\tau}X_{a_i}(\tau) &= \kappa\Upsilon(S(X_{a_i}(\tau)))\frac{\nabla_{X_{a_i}} \phi(X_{a_i}(\tau))}{\|\nabla_{X_{a_i}} \phi(X_{a_i}(\tau))\|}( p_{\max}- p(X_{a_i}(\tau), \tau) ),\\
X_{a_i}(0) &= X_{a_{i}0}, \ i\in\{1,\dots, N_A\}.
\end{cases}
\end{align}
\begin{align}\label{passive-nond4}
\begin{cases}
\frac{d}{d\tau}X_{p_k}(\tau) &= \kappa\sum_{j=1}^{N}\frac{X_{c_j} - X_{p_k} }{\epsilon + \|X_{c_j} - X_{p_k}\|}w(\|X_{c_j} - X_{p_k}\|, S(X_{p_k}(\tau)) + \kappa\beta(S(X_{p_k}(\tau)))dB(\tau),\\
X_{p_k}(0) &= X_{p_k0}, \ k\in\{1,\dots, N_P\}.
\end{cases}
\end{align}

\section*{Acknowledgment}
We thank O. M. Richardson (Karlstad), E.N.M. Cirillo (Rome) and M. Colangeli (L'Aquila) for very fruitful discussions on the topic of active-passive pedestrian dynamics through heterogeneous environments.

\vskip2mm

\bibliographystyle{siam}
\bibliography{mybibn}

\begin{thebibliography}{10}

\bibitem{Adams03}
{\sc R.~A. Adams and J.~J. Fournier}, {\em Sobolev {S}paces}, vol.~140,
  Academic Press, 2003.

\bibitem{Aurell2020}
{\sc A.~Aurell and B.~Djehiche}, {\em Behavior near walls in the mean-field
  approach to crowd dynamics}, SIAM J. Appl. Math., 80 (2020), pp.~1153--1174.

\bibitem{Bernoff2011}
{\sc A.~J. Bernoff and C.~M. Topaz}, {\em A primer of swarm equilibria}, SIAM
  J. Applied Dynamical Systems, 10 (2011), pp.~212--250.

\bibitem{Billingsley1999}
{\sc P.~Billingsley}, {\em Convergence of Probability Measures}, John Wiley \&
  Sons, Inc, 1999.

\bibitem{Bossy2004}
{\sc M.~Bossy, E.~Gobet, and D.~Talay}, {\em A symmetrized {E}uler scheme for
  an efficient approximation of reflected diffusions}, Journal of Applied
  Probability, 41 (2004), pp.~877--889.

\bibitem{Carrillo2013}
{\sc J.~A. Carrillo, S.~Martin, and V.~Panferov}, {\em A new interaction
  potential for swarming models}, Physica D: Nonlinear Phenomena, 260 (2013),
  pp.~112--126.

\bibitem{Cejkova2018}
{\sc J.~Cejkova, R.~Toth, A.~Braun, M.~Branicki, D.~Ueyama, and I.~Lagzi}, {\em
  Shortest path finding in mazes by active and passive particles}, vol.~32, in
  Adamatzky A. (eds) Shortest Path Solvers. From Software to Wetware.
  Emergence, Complexity and Computation, Springer, Cham, 2018.

\bibitem{Cholaquidis2016}
{\sc A.~Cholaquidis, R.~Fraiman, G.~Lugosi, and B.~Pateiro-L\'{o}pez}, {\em Set
  estimation from reflected {B}rownian motion}, Journal of the Royal
  Statistical Society: Series B (Statistical Methodology), 78 (2016),
  pp.~1057--1078.

\bibitem{Choulli2016}
{\sc M.~Choulli}, {\em Applications of Elliptic Carleman Inequalities to Cauchy
  and Inverse Problems}, Springer, 2016.

\bibitem{Cirillo2019}
{\sc E.~N.~M. Cirillo, M.~Colangeli, A.~Muntean, and T.~K.~T. Thieu}, {\em A
  lattice model for active-passive pedestrian dynamics: a quest for drafting
  effects}, Mathematical Biosciences and Engineering, 17 (2019), pp.~460--477.

\bibitem{Colangeli2019}
{\sc M.~Colangeli, A.~Muntean, O.~Richardson, and T.~K.~T. Thieu}, {\em
  Modelling interactions between active and passive agents moving through
  heterogeneous environments}, vol.~1: Theory, Models and Safety Problems,, in
  G. Libelli, N. Bellomo (Eds), Crowd Dynamics, Modeling and Simulation in
  Science, Engineering and Technology, Boston, Birkhauser, Springer, 2019.

\bibitem{Crisan2018}
{\sc D.~Crisan, C.~Janjigian, and T.~G. Kurtz}, {\em Particle representations
  for stochastic partial differential equations with boundary conditions},
  Electronic Journal of Probability, 23 (2018), pp.~1--29.

\bibitem{CristianiD2019}
{\sc E.~Cristiani and D.~Peri}, {\em Robust design optimization for egressing
  pedestrians in unknown environments}, Applied Mathematical Modelling, 72
  (2019), pp.~553--568.

\bibitem{Prato14}
{\sc G.~Da~Prato and J.~Zabczyk}, {\em Stochastic {E}quations in {I}nfinite
  {D}imensions}, Cambridge {U}niversity {P}ress, 2014.

\bibitem{Dupuis1993}
{\sc P.~Dupuis and H.~Ishii}, {\em Sdes with oblique reflection on nonsmooth
  domains}, Mathematical and Computer Modeling, 1 (1993), pp.~554--580.

\bibitem{Evans2013}
{\sc L.~C. Evans}, {\em An {I}ntroduction to {S}tochastic {D}ifferential
  {E}quations}, vol.~82, American Mathematical Soc., 2013.

\bibitem{Hille2016}
{\sc J.~H.~M. Evers, S.~C. Hille, and A.~Muntean}, {\em Measure-valued mass
  evolution problems with flux boundary conditions and solution-dependent
  velocities}, SIAM J. Math. Anal., 48 (2016), pp.~1929--1953.

\bibitem{Maury2015}
{\sc S.~Faure and B.~Maury}, {\em Crowd motion from the granular standpoint},
  Mathematical Models and Methods in Applied Sciences (M3AS), 25 (2015),
  pp.~463--493.

\bibitem{Flandoli95}
{\sc F.~Flandoli and D.~Gatarek}, {\em Martingale and stationary solutions for
  stochastic {N}avier-{S}tokes equations}, Probability Theory and Related
  Fields, 102 (1995), pp.~367--391.

\bibitem{Gilbarg1977}
{\sc D.~Gilbarg and N.~S. Trudinger}, {\em Elliptic Partial Differential
  Equations of Second Order}, vol.~224, Springer, 1977.

\bibitem{Horiuchi1986}
{\sc S.~Horiuchi, Y.~Murozaki, and A.~Hukugo}, {\em A case study of fire and
  evacuation in a multi-purpose office building, {O}saka, {J}apan}, Fire Safety
  Science, 1 (1986), pp.~523--532.

\bibitem{Ikeda1981}
{\sc N.~Ikeda and S.~Watanabe}, {\em Stochastic Differential Equations and
  Diffusion Processes.}, Amsterdam-Tokyo: North Holland-Kodansha, 1981.

\bibitem{Jacod2004}
{\sc J.~Jacod and P.~Protter}, {\em Probability {E}ssentials}, Springer Science
  \& Business Media, 2004.

\bibitem{Jin1997}
{\sc T.~Jin}, {\em Studies on human behavior and tenability in fire smoke},
  Fire Safety Science - Proceedings of the Fifth International Symposium, 5
  (1997), pp.~3--12.

\bibitem{Karatzars2000}
{\sc I.~Karatzars and S.~E. Shreve}, {\em Brownian Motion and Stochastic
  Calculus}, Second Edition, Graduate Texts in Mathematics, Springer, 2000.

\bibitem{Kimura2019}
{\sc M.~Kimura, P.~van Meurs, and Z.~Yang}, {\em Particle dynamics subject to
  impenetrable boundaries: {E}xistence and uniqueness of mild solutions}, SIAM
  J. Math. Anal., 51 (2019), pp.~5049--5076.

\bibitem{Lieberman1986}
{\sc G.~M. Lieberman}, {\em Mixed boundary value problems for elliptic and
  parabolic differential equations of second order}, Journal of Mathematical
  Analysis and Applications, 113 (1986), pp.~422--440.

\bibitem{Lions1984}
{\sc P.~L. Lions}, {\em Stochastic differential equations with reflecting
  boundary conditions}, Communications on Pure and Applied Mathematics, XXXVII
  (1984), pp.~511--537.

\bibitem{Lorenz2009}
{\sc J.~Lorenz}, {\em Heterogeneous bounds of confidence: Meet, discuss and
  find consensus!}, Complexity, 15 (2009), pp.~43--52.

\bibitem{Onskog2010}
{\sc K.~Nystr\"om and T.~\"Onskog}, {\em The {S}korohod oblique reflection
  problem in time-dependent domains}, The Annals of Probability, 38 (2010),
  pp.~2170--2223.

\bibitem{Groeber2014}
{\sc J.~L. P.~Groeber and F.~Schweitzer}, {\em Dissonance minimization as a
  microfoundation of social influence in models of opinion formation}, Journal
  of Mathematical Sociology, 38 (2014), pp.~147--174.

\bibitem{Mavrodiev2013}
{\sc C.~J.~T. P.~Mavrodiev and F.~Schweitzer}, {\em Quantifying the effects of
  social influence}, Scientific Reports, 3 (2013), p.~1360.

\bibitem{Pilipenko2014}
{\sc A.~Pilipenko}, {\em An Introduction to Stochastic Differential Equations
  with Reflection}, Lectures in Pure and Applied Mathematics. Institutional
  Repository of the University of Potsdam, Germany, 2014.

\bibitem{Richardson2019}
{\sc O.~Richardson, A.~Jalba, and A.~Muntean}, {\em The effect of environment
  knowledge in evacuation scenarios involving fire and smoke – a multiscale
  modelling and simulation approach}, Fire Technology, 55 (2019), pp.~415--436.

\bibitem{Ronchi2019}
{\sc E.~Ronchi and D.~Nilsson}, {\em Pedestrian movement in smoke: theory, data
  and modelling approaches}, vol.~1: Theory, Models and Safety Problems, in G.
  Libelli, N. Bellomo (Eds), Crowd Dynamics, Modeling and Simulation in
  Science, Engineering and Technology, Boston, Birkhauser, Springer, 2019.

\bibitem{Saisho1987}
{\sc Y.~Saisho}, {\em Stochastic differential equations for multi-dimensional
  domain with reflecting boundary}, Probab. Th. Rel. Fields, 74 (1987),
  pp.~455--477.

\bibitem{Schieborn:2006}
{\sc D.~Schieborn}, {\em Viscosity {S}olutions of {H}amilton-{J}acobi
  {E}quations of {E}ikonal {T}ype on {R}amified {S}paces}, PhD thesis,
  University T\"{u}bingen, Germany, 2006.

\bibitem{Skorohod1961}
{\sc A.~V. Skorohod}, {\em Stochastic equations for diffusion process in a
  bounded domain}, Theory of Probability and Its Applications, VI (1961),
  pp.~264--274.

\bibitem{Slominski1993}
{\sc L.~Slominski}, {\em On existence, uniqueness and stability of solutions of
  multidimensional sde’s with reflecting boundary conditions}, Ann. Inst.
  Henri. Poincar\'{e}, 29 (1993), pp.~163--198.

\bibitem{Pent2007}
{\sc P.~G. T.~Pent and F.~Schweitzer}, {\em Coexistence of social norms based
  on in-and out-group interactions}, Advances in Complex Systems, 10 (2007),
  pp.~271--286.

\bibitem{Thieu2019}
{\sc T.~K.~T. Thieu, M.~Colangeli, and A.~Muntean}, {\em Weak solvability of a
  fluid-like driven system for active-passive pedestrian dynamics}, Nonlinear
  Studies, 26 (2019), pp.~991--1006.

\bibitem{Yang2016}
{\sc L.~I. Yang, C.~Jianzhong, Z.~Qian, and Y.~Huizhen}, {\em Study of
  pedestrians evacuation model considering familiarity with environment}, China
  Safety Science Journal, 26 (2016), pp.~168--174.

\bibitem{Yang2014}
{\sc X.~Yang, H.~Dong, Q.~Wang, Y.~Chen, and X.~Hu}, {\em Guided crowd dynamics
  via modified social force model}, Physica A : Statistical Mechanics and its
  Applications, 411 (2014), pp.~63--73.

\end{thebibliography}

%
%
%
%
%
%
%

          \end{document}